\documentclass[11pt,reqno]{amsart}

\usepackage[T1]{fontenc}
\usepackage{lmodern}
\usepackage{microtype}
\usepackage[margin=1.03in]{geometry}
\usepackage{amsmath,amssymb,amsfonts,mathtools}
\usepackage{amsthm}
\usepackage{booktabs,tabularx,float}
\usepackage{enumitem}
\usepackage{xcolor}
\usepackage[colorlinks=true,
  linkcolor=blue!50!black,
  citecolor=blue!50!black,
  urlcolor=blue!50!black]{hyperref}
\usepackage[capitalize,noabbrev]{cleveref}

\allowdisplaybreaks
\numberwithin{equation}{section}
\setlist[enumerate]{leftmargin=2.2em,itemsep=0.25em,topsep=0.4em}
\setlist[itemize]{leftmargin=2em,itemsep=0.25em,topsep=0.4em}

\theoremstyle{plain}
\newtheorem{theorem}{Theorem}[section]
\newtheorem{proposition}[theorem]{Proposition}
\newtheorem{lemma}[theorem]{Lemma}
\newtheorem{corollary}[theorem]{Corollary}

\theoremstyle{definition}

\theoremstyle{remark}
\newtheorem{remark}[theorem]{Remark}

\newcommand{\Int}{\operatorname{Int}}
\newcommand{\Frac}{\operatorname{Frac}}
\newcommand{\F}{\mathbb F}
\newcommand{\m}{\mathfrak m}

\title[WPC Does Not Imply APC]{Weak Polynomial Completeness Does Not Imply \\ Almost Polynomial Completeness}

\author{Viet-Hoang Tran}
\address{Department of Mathematics, National University of Singapore, Singapore 119076}
\email{hoang.tranviet@u.nus.edu}
\urladdr{https://vh-tran.github.io/}

\author{Tan M. Nguyen}
\address{Department of Mathematics, National University of Singapore, Singapore 119076}
\email{tanmn@nus.edu.sg}
\urladdr{https://tanmnguyen89.github.io/}

\date{}


\hypersetup{
  pdftitle={Weak Polynomial Completeness Does Not Imply Almost Polynomial Completeness},
  pdfauthor={Viet-Hoang Tran}
}

\begin{document}
\raggedbottom

\begin{abstract}
We show that weak polynomial completeness does not imply almost
polynomial completeness by constructing an explicit extension of
domains \(D\subseteq A\) such that
\[
 \Int(D)\subseteq\Int(A)
 \qquad\text{but}\qquad
 \Int(D^2)\nsubseteq\Int(A^2).
\]
\end{abstract}

\maketitle

\section{Introduction}

Let \(R\) be a domain with quotient field \(Q\).  For \(n\geq 1\), its
ring of integer-valued polynomials in \(n\) variables is
\[
 \Int(R^n)
   =\{F\in Q[Z_1,\ldots,Z_n]:F(R^n)\subseteq R\}.
\]
Let \(R\subseteq B\) be an extension of domains and put
\[
 \Int(R,B)
   =\{f\in\Frac(B)[Z]:f(R)\subseteq B\}.
\]
The extension is \emph{polynomially complete} (PC) if \(R\) is
polynomially dense in \(B\), meaning that
\[
 \Int(R,B)=\Int(B).
\]
It is \emph{almost polynomially complete} (APC) if
\[
 \Int(R^n)\subseteq\Int(B^n)
 \qquad\text{for every }n\geq1,
\]
where the inclusion is taken through
\(Q\subseteq\Frac(B)\), and it is
\emph{weakly polynomially complete} (WPC) if
\[
 \Int(R)\subseteq\Int(B).
\]
Thus PC measures polynomial density, APC tests integer-valued
polynomials in every finite number of variables, and WPC tests only
one variable. Table~\ref{tab:hierarchy} summarizes the relationships among these three notions.

\begin{table}[H]
\centering
\footnotesize
\renewcommand{\arraystretch}{1.10}
\begin{tabularx}{0.82\textwidth}{@{}lX@{}}
\toprule
Direction & Status and reason\\
\midrule
\(\mathrm{PC}\Rightarrow\mathrm{APC}\)
& Proved by successive one-variable specialization
  \cite[Section~7]{Elliott2007}.\\
\(\mathrm{APC}\Rightarrow\mathrm{WPC}\)
& Immediate from the definition of APC by taking
  \(n=1\).\\
\(\mathrm{APC}\nRightarrow\mathrm{PC}\)
& The extension
  \(\mathbb Z[T]\subseteq\mathbb Z[T/2]\)
  is APC but not PC \cite[p.~231]{Elliott2009}.\\
\(\mathrm{WPC}\nRightarrow\mathrm{APC}\)
& Formerly open; \cref{thm:main}
  constructs a WPC extension that is not APC.\\
\bottomrule
\end{tabularx}
\caption{Known implications, counterexamples, and the contribution of
this work.}
\label{tab:hierarchy}
\end{table}

These conditions satisfy
\(\mathrm{PC}\Longrightarrow\mathrm{APC}\Longrightarrow\mathrm{WPC}\).
The second implication is the case \(n=1\).  For the first, let
\(R\subseteq B\) be PC and fix \(F\in\Int(R^n)\).  Inductively suppose
that the first \(k-1\) variables may be specialized arbitrarily in
\(B\), while the remaining variables range over \(R\).  After fixing
all variables except the \(k\)-th, the resulting
\(h\in\Frac(B)[Z]\) satisfies \(h(R)\subseteq B\).  Hence
\(h\in\Int(R,B)=\Int(B)\), so the \(k\)-th variable may also be
specialized in \(B\).  The induction starts from
\(F(R^n)\subseteq R\subseteq B\) and ends with
\(F(B^n)\subseteq B\), proving PC \(\Rightarrow\) APC
\cite[Section~7]{Elliott2007}.

The study of integer-valued polynomial rings traces back to the
classical work of Pólya and Ostrowski and has since developed into a
substantial part of commutative algebra \cite{CahenChabert1997}.
Elliott organized polynomial density and the one- and several-variable
conditions into the hierarchy summarized in \cref{tab:hierarchy}
\cite{Elliott2007}.  The converse WPC \(\Rightarrow\) APC remained
open: Elliott recorded that
no counterexample was known and later wrote that he suspected the
converse to fail \cite[p.~232]{Elliott2009}
\cite[Section~6]{Elliott2013}.  This remaining implication was listed
as Problem~22 in \emph{Open Problems in Commutative Ring Theory}
\cite{CahenFontanaFrischGlaz2014}.

The free WPC extension makes the unresolved direction precise.  Let
\(\Int_w(R^n)\) denote the smallest \(R\)-subalgebra of
\(\Int(R^n)\) that contains \(R[Z_1,\ldots,Z_n]\) and is closed under
the unary substitutions
\[
 G\longmapsto f(G),\qquad f\in\Int(R).
\]
Here \(\Int_w(R^n)\) is the free WPC extension on \(n\) generators,
whereas \(\Int(R^n)\) is the free APC extension.  Consequently, every
WPC extension of \(R\) is APC if and only if
\[
 \Int_w(R^n)=\Int(R^n)
 \qquad\text{for every }n\geq1.
\]
This is Elliott's free-object criterion
\cite[p.~232, Proposition~5.2]{Elliott2009}.  For comparison, the
canonical multiplication map \(\theta_n\) gives
\begin{equation}
 \begin{aligned}
 \theta_n:\Int(R)^{\otimes_R n}&\longrightarrow\Int(R^n),&
 f_1\otimes\cdots\otimes f_n&\longmapsto\prod_{i=1}^n f_i(Z_i),\\
 \Int_{\otimes}(R^n):=\operatorname{im}(\theta_n)
 &\subseteq\Int_w(R^n)\subseteq\Int(R^n).&&
 \end{aligned}
 \label{eq:three-closures}
\end{equation}
Surjectivity of every \(\theta_n\) proves WPC \(\Rightarrow\) APC and
yields many known positive cases \cite{Elliott2007,Elliott2009}.
Recent counterexamples concern only the first inclusion in
\eqref{eq:three-closures}; their witnesses already lie in
\(\Int_w(R^2)\) \cite{PengTaoWangYu2026,JiangEtAl2026}.  They therefore
did not decide whether the second inclusion can be strict.

Our obstruction uses characteristic \(2\): unary substitution preserves
the \(V\)-integrality of a mixed derivative, while a suitable binary
integer-valued polynomial has a mixed derivative with a forbidden
denominator.

\begin{theorem}\label{thm:main}
Let \(V\) be a discrete valuation domain of characteristic \(2\), let
\(\pi\) be a uniformizer of \(V\), and suppose that the residue field
\(L=V/\pi V\) is infinite.  Set
\[
 D=\F _2+\pi V
   =\{x\in V:x\bmod\pi V\in\F _2\}.
\]
There exists a domain \(A\) with \(D\subseteq A\) such that
\[
 \Int(D)\subseteq\Int(A)
 \quad\text{but}\quad
 \Int(D^n)\nsubseteq\Int(A^n)
 \quad\text{for every }n\geq2.
\]
One may take \(V=\F _2(s)[t]_{(t)}\), \(\pi=t\), and
\(D=\F _2+t\F _2(s)[t]_{(t)}\).
\end{theorem}

We take \(A=\Int_w(D^2)\), with variables denoted by \(U,W\), and prove
\[
 \frac{(U^2+U)(W^2+W)}{\pi}
\]
belongs to \(\Int(D^2)\setminus A\).  Thus
\[
 \Int_w(D^2)\subsetneq\Int(D^2),
\]
which is precisely the second separation in \eqref{eq:three-closures}
that the earlier tensor counterexamples did not provide.  In
particular, \(D\subseteq A\) is WPC but not APC, so
\cref{thm:main} answers the remaining converse negatively.  Combined
with the APC-but-not-PC example in \cref{tab:hierarchy}, it shows that
neither implication in the displayed hierarchy is reversible.  The present
\(D\) is non-Noetherian, so
the corresponding question for Noetherian base domains or under other
finiteness hypotheses is not decided here; see \cref{rem:scope}.

\section{The pullback domain and a unary derivative lemma}
\label{sec:pullback}

Throughout the proof, \(V\), \(\pi\), \(L\), and \(D\) have the meaning
fixed in \cref{thm:main}, and
\[
 K=\Frac(V).
\]
We regard \(\F _2\) as the prime subfield of \(L\).

\begin{lemma}\label{lem:D-structure}
The ring \(D\) is a local one-dimensional domain with maximal ideal
\(\m=\pi V\), residue field \(D/\m\cong\F _2\), and quotient field
\(\Frac(D)=K\).
\end{lemma}

\begin{proof}
The residue map \(V\to L\) shows immediately that \(D\) is a subring of
the domain \(V\), that \(\pi V\) is an ideal of \(D\), and that
\(D/\pi V\cong\F _2\).  If \(x\in D\setminus\pi V\), then its
residue is \(1\), so \(x=1+y\) for some \(y\in\pi V\).  It is a unit
of \(V\), and
\[
 x^{-1}-1=-\frac{y}{x}\in\pi V.
\]
Hence \(x^{-1}\in 1+\pi V\subseteq D\).  Every element outside
\(\pi V\) is therefore a unit of \(D\), and \(\m=\pi V\) is the unique
maximal ideal.

We next determine the nonzero prime ideals.  Let \(P\) be a nonzero
proper prime ideal of \(D\), and choose \(0\ne a\in P\).  Write
\(a=\pi^r u\), where \(r\geq0\) and \(u\in V^\times\).  The case
\(r=0\) would make \(a\) a unit of \(D\), by the preceding paragraph,
so \(r\geq1\).  Since \(\pi u^{-1}\in\pi V\subseteq D\), we have
\[
 \pi^{r+1}=a(\pi u^{-1})\in P.
\]
Primality gives \(\pi\in P\).  For an arbitrary
\(z=\pi v\in\m\), one has
\[
 z^2=\pi(\pi v^2)\in P,
\]
because \(\pi v^2\in\pi V\subseteq D\).  Thus \(z\in P\), so
\(\m\subseteq P\).  Therefore \(P=\m\).  The only prime ideals are
\((0)\) and \(\m\), and \(\dim D=1\).

Finally, \(\pi\in D\) and \(\pi v\in D\) for every \(v\in V\).
Consequently
\[
 v=\frac{\pi v}{\pi}\in\Frac(D),
\]
so \(V\subseteq\Frac(D)\).  It follows that
\(K=\Frac(V)\subseteq\Frac(D)\).  The reverse inclusion follows from
\(D\subseteq V\subseteq K\), proving \(\Frac(D)=K\).
\end{proof}

The infinitude of the residue field has the following standard
consequence.  We include the argument because it is the precise point
at which the hypothesis on \(L\) enters.

\begin{lemma}\label{lem:IntV}
If \(V\) is a discrete valuation domain with infinite residue field,
then
\[
 \Int(V)=V[Z].
\]
\end{lemma}

\begin{proof}
The inclusion \(V[Z]\subseteq\Int(V)\) is immediate.  Conversely, let
\(H=\sum_{i=0}^d b_iZ^i\in K[Z]\) satisfy \(H(V)\subseteq V\).
Assume that some coefficient is not in \(V\).  With \(v\) denoting the
normalized valuation of \(V\), put
\[
 r=-\min_{b_i\ne0}v(b_i)>0.
\]
Then \(G=\pi^rH\in V[Z]\), and at least one coefficient of \(G\) is a
unit.  Thus its reduction \(\overline G\in L[Z]\) is nonzero.
Nevertheless, for every \(x\in V\),
\[
 G(x)=\pi^rH(x)\in\pi^rV\subseteq\pi V.
\]
Every element of \(L\) has a lift in \(V\), so \(\overline G\) vanishes
at every element of the infinite field \(L\).  A nonzero polynomial
over a field has at most its degree many roots, a contradiction.
Hence every \(b_i\in V\), and \(H\in V[Z]\).
\end{proof}

\begin{lemma}[Unary derivative lemma]\label{lem:derivative}
For every \(f\in\Int(D)\) and every \(a\in D\), one has
\[
 f'(a)\in V.
\]
\end{lemma}

\begin{proof}
Fix \(f\in\Int(D)\) and \(a\in D\), and introduce the polynomial
\[
 H_a(T)=f(a+\pi T)\in K[T].
\]
For every \(t\in V\), the element \(a+\pi t\) lies in \(D\); hence
\[
 H_a(t)=f(a+\pi t)\in D\subseteq V.
\]
Thus \(H_a\in\Int(V)\), and \cref{lem:IntV} gives
\(H_a\in V[T]\).

Let \(\overline H_a\in L[T]\) denote its coefficientwise reduction.
For each \(\lambda\in L\), choose a lift \(t\in V\).  Since
\(H_a(t)\in D\), its residue lies in the distinguished subfield
\(\F _2\subseteq L\).  Therefore
\[
 \overline H_a(\lambda)^2+\overline H_a(\lambda)=0
 \qquad(\lambda\in L).
\]
The polynomial \(\overline H_a^2+\overline H_a\) vanishes on the
infinite field \(L\), so it is the zero polynomial.  Since \(L[T]\) is
an integral domain,
\[
 \overline H_a(\overline H_a+1)=0
 \quad\Longrightarrow\quad
 \overline H_a=0\ \text{or}\ \overline H_a=1.
\]
In particular, every positive-degree coefficient of \(H_a\) belongs to
\(\pi V\).

The coefficient of \(T\) in \(H_a(T)=f(a+\pi T)\) is
\(\pi f'(a)\).  It follows that \(\pi f'(a)\in\pi V\).  Thus
\(\pi f'(a)=\pi v\) for some \(v\in V\), and cancellation of the
nonzero element \(\pi\) in the field \(K\) yields \(f'(a)=v\in V\).
\end{proof}

\begin{remark}\label{rem:constant-reduction}
The proof actually shows more: for every \(f\in\Int(D)\) and
\(a\in D\), the reduction of \(f(a+\pi T)\) modulo \(\pi\) is a
constant in \(\F _2\).  We shall only need the resulting control of the
first derivative.
\end{remark}

\section{Unary closure and the mixed-derivative invariant}
\label{sec:closure}

Let \(U,W\) be algebraically independent over \(K\), and set
\[
 C=\Int(D^2)\subseteq K[U,W].
\]
We recursively define \(D\)-subalgebras of \(C\).  Put
\(A_0=D[U,W]\), and, once \(A_r\) is defined, let
\begin{equation}
 A_{r+1}
  =D\big[A_r,\ f(G):f\in\Int(D),\ G\in A_r\big]\subseteq C.
 \label{eq:Ar}
\end{equation}
Every polynomial in \(D[U,W]\) maps \(D^2\) into \(D\), so
\(A_0\subseteq C\).  The displayed containment is also legitimate.
Indeed, if \(G\in C\) and
\((x,y)\in D^2\), then \(G(x,y)\in D\); consequently
\(f(G(x,y))\in D\) for every \(f\in\Int(D)\).  Thus \(f(G)\in C\),
and induction proves that every \(A_r\) is contained in \(C\).  Define
\begin{equation}
 A=\bigcup_{r\geq0}A_r.
 \label{eq:A}
\end{equation}
The sequence is ascending, so \(A\) is a \(D\)-subalgebra of
\(K[U,W]\).

\begin{proposition}\label{prop:WPC}
The ring \(A\) is a domain,
\[
 \Frac(A)=K(U,W),
\]
and the extension \(D\subseteq A\) is WPC.
\end{proposition}

\begin{proof}
Because \(A\subseteq K[U,W]\), it is a domain.  The inclusions
\[
 D[U,W]\subseteq A\subseteq K[U,W]
\]
imply, after taking quotient fields,
\[
 K(U,W)=\Frac(D[U,W])
 \subseteq\Frac(A)\subseteq K(U,W).
\]
Here \(\Frac(D)=K\) by \cref{lem:D-structure}.  Hence
\(\Frac(A)=K(U,W)\).

Let \(f\in\Int(D)\) and \(G\in A\).  There is some \(r\) with
\(G\in A_r\), and then the definition \eqref{eq:Ar} gives
\(f(G)\in A_{r+1}\subseteq A\).  Therefore \(f(A)\subseteq A\).
Since \(K\subseteq K(U,W)=\Frac(A)\), we may regard \(f\) as an
element of \(\Frac(A)[Z]\), and the preceding inclusion says exactly
that \(f\in\Int(A)\).  Hence
\(\Int(D)\subseteq\Int(A)\).
\end{proof}

\begin{remark}\label{rem:free}
The iteration is the concrete unary-closure construction underlying
free WPC algebras in Elliott's approach
\cite{Elliott2009,Elliott2010}.  More precisely, \(A\) is the smallest
unary-closed \(D\)-subalgebra of \(K[U,W]\) containing \(D[U,W]\):
it is unary-closed by \cref{prop:WPC}, and every other such subalgebra
contains each \(A_r\), by induction on \(r\), and hence contains \(A\).
Thus, in the notation of the introduction,
\[
 A=\Int_w(D^2).
\]
\end{remark}

We now isolate an invariant of the construction.  For
\(G\in K[U,W]\), write \(G_U,G_W,G_{UW}\) for its formal partial
derivatives.  Define
\begin{equation}
 E=\left\{
 G\in C:
 G_U(0,0),\ G_W(0,0),\ G_{UW}(0,0)\in V
 \right\}.
 \label{eq:E}
\end{equation}
Notice that every \(G\in C\) satisfies \(G(0,0)\in D\), since
\((0,0)\in D^2\).

\begin{proposition}\label{prop:invariant}
The set \(E\) is a \(D\)-subalgebra of \(C\), it contains
\(D[U,W]\), and it is closed under every unary operation
\[
 G\longmapsto f(G),\qquad f\in\Int(D).
\]
Consequently,
\[
 A\subseteq E.
\]
\end{proposition}

\begin{proof}
Closure under addition and multiplication by elements of \(D\) follows
directly from linearity of formal differentiation and \(D\subseteq V\).
Let \(G,H\in E\).  The product rules give
\[
 (GH)_U=G_UH+GH_U,\qquad
 (GH)_W=G_WH+GH_W
\]
and
\begin{equation}
 (GH)_{UW}
  =G_{UW}H+G_UH_W+G_WH_U+GH_{UW}.
 \label{eq:product-mixed}
\end{equation}
At the origin, \(G(0,0)\) and \(H(0,0)\) lie in \(D\subseteq V\),
while all derivatives occurring on the right sides lie in \(V\) by
the definition of \(E\).  Every displayed derivative of \(GH\) at the
origin therefore belongs to \(V\).  This proves closure under
multiplication.  Closure under additive inverses also follows from
linearity of differentiation.

For a polynomial
\(G=\sum_{i,j}a_{ij}U^iW^j\in D[U,W]\), the three derivatives at the
origin are \(a_{10}\), \(a_{01}\), and \(a_{11}\), respectively (with
a missing coefficient interpreted as \(0\)).  In particular, they
belong to \(D\subseteq V\).
Thus
\[
 D[U,W]\subseteq E.
\]
In particular, \(E\) contains \(D\); together with the preceding
closure properties, this proves that \(E\) is a \(D\)-subalgebra of
\(C\).

It remains to prove unary closure.  Take \(G\in E\) and
\(f\in\Int(D)\), and put \(a=G(0,0)\in D\).  The formal chain rule
gives
\begin{equation}
 \begin{aligned}
 (f(G))_U&=f'(G)G_U,\\
 (f(G))_W&=f'(G)G_W,\\
 (f(G))_{UW}
   &=f''(G)G_UG_W+f'(G)G_{UW}.
 \end{aligned}
 \label{eq:chain}
\end{equation}
The base field has characteristic \(2\).  If
\(f(Z)=\sum_j c_jZ^j\), then
\[
 f''(Z)=\sum_{j\geq2}j(j-1)c_jZ^{j-2}=0,
\]
because every integer \(j(j-1)\) is even.  Evaluating
\eqref{eq:chain} at the origin therefore yields
\[
 \begin{aligned}
 (f(G))_U(0,0)&=f'(a)G_U(0,0),\\
 (f(G))_W(0,0)&=f'(a)G_W(0,0),\\
 (f(G))_{UW}(0,0)&=f'(a)G_{UW}(0,0).
 \end{aligned}
\]
By \cref{lem:derivative}, \(f'(a)\in V\); by hypothesis, the other
factors lie in \(V\).  All three products lie in \(V\), and so
\(f(G)\) satisfies the derivative conditions in \eqref{eq:E}.
Moreover \(G\in C\) implies \(f(G)\in C\), as observed immediately
after \eqref{eq:Ar}.  Hence \(f(G)\in E\).

Since \(A_0=D[U,W]\subseteq E\), the unary closure just proved and the
fact that \(E\) is a \(D\)-algebra imply successively that
\(A_r\subseteq E\) for every \(r\).  Taking the union gives
\(A\subseteq E\).
\end{proof}

\section{The separating polynomial}
\label{sec:separation}

Put
\[
 q(Z)=Z^2+Z\in D[Z].
\]
Its two relevant features are that \(q(D)\subseteq\pi V\) and that
\(q'(Z)=1\).

\begin{lemma}\label{lem:q}
For every \(x\in D\), one has \(q(x)\in\pi V\).
\end{lemma}

\begin{proof}
Write \(x=\varepsilon+\pi v\), with
\(\varepsilon\in\F _2\) and \(v\in V\).  Since the characteristic is
\(2\),
\[
 \begin{aligned}
 q(x)
  &=(\varepsilon+\pi v)^2+(\varepsilon+\pi v)\\
  &=(\varepsilon^2+\varepsilon)+\pi^2v^2+\pi v
    =\pi^2v^2+\pi v\in\pi V.
 \end{aligned}
\]
Here \(\varepsilon^2+\varepsilon=0\) for
\(\varepsilon\in\F _2\).
\end{proof}

\begin{proposition}\label{prop:separation}
The polynomial
\begin{equation}
 P(U,W)=\frac{q(U)q(W)}{\pi}
       =\frac{(U^2+U)(W^2+W)}{\pi}
 \label{eq:P}
\end{equation}
belongs to \(\Int(D^2)\) but not to \(A\).
\end{proposition}

\begin{proof}
Let \(x,y\in D\).  By \cref{lem:q}, there are \(v,w\in V\) such
that \(q(x)=\pi v\) and \(q(y)=\pi w\).  Therefore
\[
 P(x,y)=\frac{(\pi v)(\pi w)}{\pi}
       =\pi vw\in\pi V\subseteq D.
\]
This proves \(P\in\Int(D^2)=C\).

On the other hand, \(q'(Z)=2Z+1=1\).  Differentiating
\eqref{eq:P} once in each variable gives
\[
 P_{UW}(U,W)=\frac{q'(U)q'(W)}{\pi}=\frac1\pi.
\]
In particular, \(P_{UW}(0,0)=1/\pi\notin V\), because \(\pi\) is a
nonunit of the valuation domain \(V\).  Thus \(P\notin E\).  Since
\(A\subseteq E\) by \cref{prop:invariant}, it follows that
\(P\notin A\).
\end{proof}

\begin{proof}[Proof of \cref{thm:main}]
Construct \(A\) by \eqref{eq:Ar}--\eqref{eq:A}.  By
\cref{prop:WPC}, it is a domain containing \(D\) and
\[
 \Int(D)\subseteq\Int(A).
\]

Fix \(n\geq2\), and in \(K[Z_1,\ldots,Z_n]\) consider
\begin{equation}
 \Phi_n(Z_1,\ldots,Z_n)
   =\frac{q(Z_1)q(Z_2)}{\pi}.
 \label{eq:Phi}
\end{equation}
The polynomial ignores its last \(n-2\) arguments.  The calculation in
the first paragraph of the proof of \cref{prop:separation} shows that
\(\Phi_n(D^n)\subseteq D\), and hence
\(\Phi_n\in\Int(D^n)\).

Suppose, for contradiction, that
\(\Phi_n\in\Int(A^n)\).  Since \(U,W,0\in A\), evaluation at the
\(n\)-tuple
\[
 (U,W,0,\ldots,0)\in A^n
\]
would give
\[
 \Phi_n(U,W,0,\ldots,0)
    =\frac{q(U)q(W)}{\pi}=P(U,W)\in A.
\]
This contradicts \cref{prop:separation}.  Therefore
\(\Phi_n\notin\Int(A^n)\).  We have proved
\[
 \Int(D^n)\nsubseteq\Int(A^n)
 \qquad\text{for every }n\geq2.
\]

For the explicit instance, take
\[
 V=\F _2(s)[t]_{(t)}.
\]
This is a discrete valuation domain of characteristic \(2\), with
uniformizer \(t\), quotient field \(\F _2(s)(t)\), and infinite residue
field \(\F _2(s)\).  All the hypotheses are therefore satisfied.
\end{proof}

\begin{corollary}\label{cor:hierarchy}
For extensions of domains, both implications in the hierarchy
\[
 \mathrm{PC}\Longrightarrow\mathrm{APC}
 \Longrightarrow\mathrm{WPC}
\]
are strict.
\end{corollary}

\begin{proof}
The implications were proved in the introduction.  Elliott's extension
\(\mathbb Z[T]\subseteq\mathbb Z[T/2]\) is APC but not PC
\cite[p.~231]{Elliott2009}, so the first implication cannot be
reversed.  The extension \(D\subseteq A\) in \cref{thm:main} is WPC
but not APC, so the second implication cannot be reversed.
\end{proof}

\begin{remark}[Scope of the example]\label{rem:scope}
The domains \(D\) occurring in \cref{thm:main} are non-Noetherian.
Indeed, \(\m=\pi V\) and
\[
 \m^2=\pi^2V.
\]
The inclusion \(\m^2\subseteq\pi^2V\) follows by multiplying elements
of \(\pi V\); conversely, for \(v\in V\), the equality
\(\pi^2v=\pi(\pi v)\) writes every element of \(\pi^2V\) as a product
of two elements of \(\m\).

Now consider the map
\[
 V\longrightarrow\m/\m^2,\qquad
 v\longmapsto\pi v+\pi^2V.
\]
It is surjective because every element of \(\m\) has the form
\(\pi v\).  Its kernel consists exactly of the \(v\in V\) such that
\(\pi v\in\pi^2V\), namely the elements of \(\pi V\).  It therefore
induces an isomorphism
\[
 L=V/\pi V\xrightarrow{\ \sim\ }\m/\m^2
\]
of vector spaces over
\(D/\m\cong\F _2\).  Since \(L\) is an infinite field containing
the finite field \(\F _2\), it is infinite-dimensional over
\(\F _2\).  Hence \(\m/\m^2\) is infinite-dimensional; if \(\m\)
were finitely generated as a \(D\)-ideal, this quotient would be
finite-dimensional.  Thus \(D\) is not Noetherian.  The construction
settles the unrestricted implication, but it does not decide what may
happen after imposing Noetherianity or other finiteness hypotheses.
\end{remark}

\renewcommand{\sectionname}{}
\section*{Acknowledgments}
We used GPT-5.6 Sol and an agentic harness built around GPT-5.6 Sol and Claude Fable 5 to assist with literature searches, hypothesis testing, the exploration and elimination of potential approaches, wording refinement, and manuscript proofreading. We thank Hieu M. Vu, Tho Tran Huu, Khoi M. N. Nguyen, Dung V. Nguyen, and Quang X. Nguyen for their assistance with hardware-related matters and for providing technical support in the use of the AI tools and agentic system. Viet-Hoang Tran thanks Thieu N. Vo for insightful discussions regarding the problem.

\bibliographystyle{plain}
\begingroup
\footnotesize
\raggedright
\bibliography{references}
\endgroup

\end{document}